\documentclass[12pt]{amsart}

\usepackage{amssymb}

\usepackage{url}
\usepackage[all]{xy}

\usepackage[colorlinks=true,linkcolor=blue,urlcolor=blue,citecolor=blue,pagebackref]{hyperref} 

\usepackage[usenames, dvipsnames, svgnames]{xcolor}

\usepackage{enumerate} % for using different counters
\usepackage{helvet} % for Sans Serif fonts
\usepackage{courier} % default ttdefault
\usepackage{eucal} % better calligraphic fonts
\usepackage{mathrsfs} % for fancy calligraphic fonts

\reversemarginpar %marginal notes appear on the left side margin
\normalmarginpar %switches back to right side margin

\let\oldmarginpar\marginpar %changes the font of margin text
\renewcommand\marginpar[1]{\-\oldmarginpar{\raggedright\small\sf #1}}

\title{Finite group schemes of essential dimension one}

\author{Najmuddin Fakhruddin}
 
\address{School of Mathematics, Tata Institute of Fundamental
  Research, Homi Bhabha Road, Mumbai 400005, India}

\email{naf@math.tifr.res.in}



\newcommand{\nc}{\newcommand}

\nc{\rnc}{\renewcommand}

\nc{\bs}{\backslash}
\nc{\te}{\otimes}
\nc{\lf}{\lfloor} %for round down
\nc{\rf}{\rfloor}
\nc{\lc}{\lceil}  %for round up
\nc{\rc}{\rceil}
\nc{\lr}{\longrightarrow}
\nc{\sr}{\stackrel}
\nc{\dar}{\dashrightarrow}
\nc{\thra}{\twoheadrightarrow}

\nc{\la}{\langle}
\nc{\ra}{\rangle} 

\nc{\ms}{\mathscr}
\nc{\mc}{\mathcal}
\nc{\mb}{\mathbb}
\nc{\mf}{\mathbf}
\nc{\mr}{\mathrm}
\nc{\mg}{\mathfrak}

\nc{\bP}{\mathbb{P}}
\rnc{\P}{\mathbb{P}}
\nc{\Q}{\mathbb{Q}}
\nc{\Z}{\mathbb{Z}}
\nc{\C}{\mathbb{C}}
\nc{\R}{\mathbb{R}}
\nc{\A}{\mathbb{A}}
\nc{\V}{\mathbb{V}}
\nc{\W}{\mathbb{W}}
\nc{\N}{\mathbb{N}}
\nc{\D}{\mathbb{D}}
\nc{\G}{\mathbb{G}}
\nc{\F}{\mathbb{F}}
\nc{\qb}{\overline{\mathbb{Q}}}

\nc{\del}{\partial}

\nc{\wt}{\widetilde}
\nc{\wh}{\widehat}
\nc{\ov}{\overline}
\nc{\un}{\underline}

\nc{\aff}{{\A}^1}
\nc{\naive}{\!\sim_n}

\nc{\Spec}{\mr{Spec}}

\nc{\omx}{\omega_X}

\nc{\ep}{\epsilon}
\nc{\ve}{\varepsilon}
\nc{\vt}{\vartheta}

\rnc{\l}{\lambda}
\rnc{\k}{\kappa}

\nc{\ovl}{\ov{\lambda}}
\nc{\vl}{\mb{V}_{\ovl}}
\nc{\dl}{\mb{D}_{\ovl}}
\nc{\mnb}{\ov{\mr{M}}_{0,n}}
\nc{\mn}{\mr{M}_{0,n}}
\nc{\mel}{\ov{\mr{M}}_{1,1}}
\nc{\mfb}{\ov{\mr{M}}_{0,4}}
\nc{\mof}{\mr{M}_{0,4}}
\nc{\mgnb}{\ov{\mr{M}}_{g,n}}
\nc{\mgn}{\ov{\mr{M}}_{g,n}}
\nc{\omc}{\ov{\mr{M}}}

\rnc{\sl}{\shoveleft}

\nc{\res}{\operatorname{Res}}
\nc{\pic}{\operatorname{Pic}}
\nc{\spec}{\operatorname{Spec}}
\nc{\im}{\operatorname{Im}}
\nc{\gal}{\operatorname{Gal}}
\nc{\fr}{\operatorname{Fr}}
\nc{\ed}{\operatorname{ed}}
\nc{\rank}{\operatorname{rank}}
\nc{\h}{\operatorname{H}}
\nc{\ch}{\operatorname{char}}
\nc{\sw}{\operatorname{sw}}
\nc{\rsw}{\operatorname{rsw}}
\nc{\supp}{\operatorname{supp}}

\nc{\Mor}{\operatorname{Mor}}
\nc{\Per}{\operatorname{Per}}
\nc{\prep}{\operatorname{Prep}}
\nc{\End}{\operatorname{End}}
\nc{\Orb}{\operatorname{Orb}}
\nc{\lie}{\operatorname{Lie}}
\nc{\aut}{\operatorname{Aut}}

\newtheorem{thm}{Theorem}[section]
\newtheorem{prop}[thm]{Proposition}

\newtheorem{lem}[thm]{Lemma}

\theoremstyle{definition}

\newtheorem{ques}[thm]{Question}
\newtheorem{rem}[thm]{Remark}

\newtheorem*{claim*}{Claim}

\newtheorem{ex}[thm]{Example}

\numberwithin{equation}{section}

\begin{document}

\begin{abstract}
  We prove that if a finite group scheme $G$ over a field $k$ has
  essential dimension one, then it embeds in $PGL_{2/k}$. We use this
  to give an explicit classification of all infinitesimal group
  schemes of essential dimension one over any field and a
  characterisation of all finite group schemes of essential dimension
  one over algebraically closed fields.
\end{abstract}

\maketitle

\section{Introduction}

Let $G$ be a group scheme over a field $k$. The essential dimension of
$G$---see e.g.~\cite{reichstein-icm} or \cite{merk-survey} for the
definition---is a non-negative integer defined using $G$-torsors over
all extensions of $k$. For example, if $G = GL_{n/k}$ then all $G$
torsors are trivial and the essential dimension is zero. The essential
dimension of a non-trivial finite group scheme $G$ is always positive,
so those having essential dimension one are of particular
interest. The main aim of this note is to prove the following:
\begin{thm} \label{thm:edone} %
  Let $G$ be a finite group scheme over a field $k$.
  \begin{enumerate}
  \item If the essential dimension of $G$ over $k$ is one then $G$ can
    be embeded in $PGL_{2/k}$ and $\dim_k(\lie(G)) \leq 1$.
  \item If $G$ is infinitesimal then it has essential dimension one
    over $k$ iff it can be embedded in $PGL_{2/k}$,
    $\dim_k(\lie(G)) = 1$, and $G$ lifts to $GL_{2/k}$. Such group
    schemes exist over $k$ iff $\ch(k) =p > 0$ and a list of such
    group schemes is as follows:
    \begin{enumerate}
    \item $\alpha_{p^n}$ for all $n>0$;
    \item $\mu_{p^n}$ for all $n>0$;
    \item any form of $\mu_{p^n}$, $n > 0$, which becomes isomorphic
      to $\mu_{p^n}$ over a quadratic extension of $k$ if $p \neq 2$.
    \end{enumerate}
  \end{enumerate}
\end{thm}
Ledet proved in \cite{ledet-one} that a constant finite group $G$ has
essential dimension one iff it embeds in $PGL_{2/k}$ and lifts to
$GL_{2/k}$.  A complete list of such $G$ was given by Chu, Hu, Kang
and Zhang \cite{CHKZ}. It is likely that by combining these results
with Theorem \ref{thm:edone} one can classify all finite group schemes
of essential dimension one.  Theorem \ref{thm:finite} gives such a
classification for group schemes over perfect fields with constant
etale quotients.

Tossici has conjectured \cite[Conjecture 1.4]{toss-unip} that for a
finite commutative unipotent group $G$, $\ed_k(G) \geq n_V(G)$, where
$n_V(G)$ is the order of nilpotence of the Verschiebung morphism of
$G$. Theorem \ref{thm:edone}, and the fact that $\ed_k(\Z/p^2\Z) = 2$
(over a field of characteristic $p$), implies that $\ed_k(G) \geq 2$
if $n_V(G) \geq 2$. In particular, we see that the $p$-torsion of a
supersingular elliptic curve has essential dimension two (Example
\ref{ex:ss}).

\smallskip

For a constant group scheme $G$ with essential dimension one over $k$
it is immediate that $G$ embeds in $PGL_{2/k}$; this is because any
rationally defined action of a finite group on a smooth projective
curve extends to a regular action. However, this extension property is
far from being true for infinitesimal group actions and the key to our
classification is a simple criterion (Proposition \ref{prop:extend})
for the existence of such extensions. Aside from this, we also use
some basic structure theory of finite group schemes, especially in the
case $\ch(k) = 2$.

\subsection{}

For the basic definitions in the theory of essential dimension we
refer the reader to \cite{reichstein-icm} or \cite{merk-survey}; we
only need the definition of essential dimension of a group scheme over
a field $k$, denoted by $\ed_k(G)$ below, the $p$-essential dimension
for a prime $p$, denoted by $\ed_k(G;p)$ and the notion of versal
torsors \cite[\S 3d]{merk-survey}.  For the particular case of
infinitesimal group schemes the reader may consult \cite{tos-vis}. Our
basic reference for the theory of (finite) group schemes is \cite{DG}.

\subsection{Acknowledgements}
I thank Patrick Brosnan and Zinovy Reichstein for helpful comments on
an earlier version of this note, and Dajano Tossici for pointing out
an error in the original proof of the main theorem when $p=2$ and for
other useful comments.

\section{Preliminary results}

\subsection{An extension criterion}

\begin{lem} \label{lem:int}%
  Let $\iota:A \to B$ be an inclusion of noetherian integral domains
  with $A$ normal and with the map $\spec(B) \to \spec(A)$ being
  surjective (or with image containing all height one primes). If
  $b \in B$ is such that $b \in A_{(0)}$ (the quotient field of $A$),
  then $b \in A$.
\end{lem}

\begin{proof}
  Since A is normal, it suffices to show that the valuation of $b$ at
  any height one prime of $A$ is non-negative, so by localisation we
  may assume that $A$ is a dvr.  If $b \notin A$ then
  $b^{-1} \in m_A$. By surjectivity of the map
  $\spec(B) \to \spec(A)$, it follows that $b^{-1} \in P$ for some
  prime ideal $P$ of $B$. But this implies that
  $1 = b\cdot b^{-1} \in P$, a contradiction.
\end{proof}

\begin{prop} \label{prop:extend}%
  Let $k$ be a field, let $G$ be an infinitesimal group scheme over
  $k$ and let $Y$ be a normal projective curve with a generically
  defined action of $G$. If there exists $X$, a normal projective
  variety over $k$ with a regular action of $G$, and a dominant
  rational map $f:X \dar Y$ compatible (generically) with the
  $G$-actions, then the rational action of $G$ extends uniquely to a
  regular action on $Y$.  If $k$ is perfect or $Y$ is smooth, then $G$
  can be taken to be any finite group scheme.
\end{prop}

It is easy to see that rational actions of infinitesimal group schemes
on smooth projective curves do not always extend to regular actions,
so our hypothesis on the existence of the equivariant rational map
$f:X\dar Y$ is not superfluous; see Example \ref{ex:ex} below.

\begin{proof}
  Assume first that $k$ is arbitrary and $G$ is infinitesimal.
  Let $U$ be the maximal open subset of $X$ on which the rational map
  $f$ restricts to a morphism $f|_U:U \to Y$. The map $f|_U$ is
  surjective, because by normality $X \bs U$ is of codimension at
  least two in $X$ so a general complete intersection curve $C \subset
  X$ will lie in $U$ and $f|_C:C \to Y$ is surjective since $C$ is
  proper and also general.

  Let $y \in Y$ be any closed point and let $u \in U$ be a (closed)
  point such that $f(u) = y$. Let $V = \spec(A)$ be any open affine
  subset of $Y$ with $y \in V$ and such that the $G$-action on $Y$ is
  defined at all points of $V \bs \{y\}$. Let $W \subset U$ be any
  affine open subset such that $u \in W$ and $f(W) \subset V$. Since
  $G$ is infinitesimal, $W$ is $G$-invariant and $G$ acts generically
  on $V$.  Let $W = \spec(B)$, so $f$ induces an inclusion
  $\iota:A \to B$.

  Let $G = \spec(R)$.  Then $R$ is a Hopf algebra and the $G$-actions
  correspond to maps $c_1:B \to R \otimes_k B $ and
  $c_2:A_{(0)} \to R \otimes A_{(0)}$ (satisfying the usual
  properties). For any $a \in A$, we have that
  $c_1(\iota(a)) = c_2(a)$, where the equality holds in
  $R \otimes B_{(0)}$. Choosing a basis of $R$ over $k$ and then
  applying Lemma \ref{lem:int} coordinatewise, we see that
  $c_2(a) \in R \otimes A$.  This proves that the action of $G$ on
  $V \bs \{y\}$ extends to all of $V$---the necessary identities hold
  because the map $A \to A_{(0)}$ is an injection---and this extension
  is clearly unique.  Since $y \in Y$ was arbitrary, it follows that
  the generically defined action of $G$ extends uniquely to all of
  $Y$.

  Now suppose that $k$ is perfect and $G$ is any finite group scheme
  over $k$. By \cite[II, \S 5, 2.4]{DG}, $G$ is a semidirect product
  of its identity component $G^0$ and its etale quotient $G^{et}$. The
  action of $G^0$ extends to all of $Y$ and the generic action of an
  etale group scheme (over any field) on a normal projective curve
  extends uniquely to a regular action. This is clear for a constant
  group scheme and the case of an etale group scheme follows from this
  by Galois descent.  The uniqueness of the two extensions then
  implies that the action of $G$ also extends.

  Now suppose that $k$ is arbitrary, $G$ is any finite group scheme,
  and $Y$ is smooth. We know that an extension of the $G^0$ action
  exists over $k$ and an extension of the $G$ action exists after we
  base change to any perfect field $K/k$, since $Y_K$ is also
  smooth. For $V = \spec(A)$ any nonempty affine open subset of $Y$ as
  above, consider the action map $c_2: A_{(0)} \to R \otimes A_{(0)}$
  as above. Using a basis of $R$ over $k$ and writing $c_2(a)$ in
  terms of coordinates, we see that an extension of the $G$ action
  actually exists over $k$ since any element of $A_{(0)}$ which lies
  in $A\otimes_kK$ must also lie in $A$. The extension is unique since
  this holds over $K$.
\end{proof}

\begin{rem}
  The assumption that $Y$ is a curve was only used in the proof to
  deduce the surjectivity of $f|_U$. If instead of $f$ being a
  rational map, we assume that it is regular and surjective (in
  codimension one), then the proof in the case of infinitesimal $G$
  works for any normal $Y$ (without any projectivity/properness
  hypothesis on $Y$ or $X$).
\end{rem}

\begin{ex} \label{ex:ex}%
  Let $Z$ be a smooth projective ordinary curve of genus $g \geq 2$
  over an algebraically closed field $k$ of charactersitic $p>0$. Let
  $T$ be a non-trivial $\mu_p$-torsor over $Z$ corresponding to
  non-zero element of $\pic(Z)[p]$. Such a torsor is non-trivial over
  the generic point of $Z$, so $T$ must be integral at its generic
  point. However, there is no $\mu_p$-equivariant completion of this
  torsor with total space a smooth projective curve since such a curve
  would map to $Z$, so would have genus at least $2$, and a smooth
  projective curve of genus $g\geq 2$ has no non-zero vector fields.
\end{ex}

We do not know the answer to the following:
\begin{ques} \label{ques:ex}%
  Does the action of a finite group scheme on a function field of
  transcendence degree one always extend to a proper model of the
  function field?
\end{ques}

\begin{lem} \label{lem:action}%
  Let $k$ be any field and $G$ an affine group scheme of finite type
  over $k$. Then there exists a smooth projective (geometrically
  irreducible) rational variety $X$ over $k$ with a generically free
  action of $G$.
\end{lem}

\begin{proof}
  Since $G$ is of finite type it can be embedded in $GL_n$ for some
  $n>0$, so $G$ acts linearly and generically freely on $M_n$, the
  space of $n \times n$ matrices. We may then take $X$ to be the
  projective completion of $M_n$, i.e., $\P(M_n \oplus k)$.
\end{proof}

\subsection{Infinitesimal group schemes with one dimensional Lie
  algebra}

In this section we assume that $\ch(k) = p > 0$.  We denote the Lie
algebra of group schemes $G$, $H$, \dots, by $\mg{g}$, $\mg{h}$,
\dots.

\begin{lem} \label{lem:sub} Let $G$ be an infinitesimal group scheme
  over a field $k$ with $\dim_k(\mg{g}) = 1$. If $H$ is any
  subquotient of $G$, then $\dim_k(\mg{h}) \leq 1$.
\end{lem}

\begin{proof}
  The statement is clear for subgroup schemes so it suffices to prove
  it for quotients. Let $K \subset G$ be a normal subgroup scheme and
  let $H = G/K$. Clearly $\mc{O}(G) \cong B = k[x]/(x^{p^n})$ for some
  $n\geq 0$, and $\mc{O}(H) \cong A$, where $A$ is a $k$-subalgebra of
  $B$ such that $B$ is flat over $A$. If $A=k$ the lemma is clear, so
  we may assume that the maximal ideal $m_A$ of $A$ is nonzero.

  Let $r$ be the smallest integer such that $m_A$ contains an element
  $a \in m^r_B\bs m^{r+1}_B$. Then the elements $1,x,\dots,x^{r-1}$ of
  $B$ are linearly independent in $B/m_AB$, so they must be part of a
  basis of $B$ over $A$. In particular, we get that
  $p^n = \dim_k(B) \leq r \dim_k(A)$. On the other hand, the structure
  of the ring $B$ implies that
  $1,a,\dots, a^{\lceil \frac{p^n}{r} \rceil}$ are $k$-independent
  elements of $A$. This implies that $r | p^n$ and $a$ generates
  $m_A$, so $\mg{h}$ is one dimensional.
  
\end{proof}

\begin{rem}
  The lemma holds with $1$ replaced by $n$ for arbitrary finite type
  group schemes $G$ over a field $k$, but we do not give the details
  here since we do not need this.
\end{rem}

\begin{prop} \label{prop:lie}%
  Let $G$ be an infinitesimal group scheme over a field $k$ with
  $\dim_k(\mg{g}) = 1$. Then either $G$ is multiplicative or it is
  unipotent.
\end{prop}

\begin{proof}
Since Lie algebras are compatible with field extensions and so are the
notions of multiplicative and unipotent groups, we may assume that $k$
is algebraically closed.
  
Let $G_1$ be the height one subgroup scheme of $G$ corresponding to
$\mg{g}$ (viewed as $p$-Lie algebra); see for example \cite[II, \S 7,
4.3]{DG}. This is a characteristic subgroup scheme of $G$ of order
$p$, so we may form the quotient $G/G_1$. By Lemma \ref{lem:sub} and
induction on the order of $G$, we get a filtration
\[
  \{1\} = G_0 \subset G_1 \subset G_2 \subset \dots \subset G_n = G
\]
where each $G_i$ is a characteristic subgroup of $G$ and $G_i/G_{i-1}$
is of order $p$ for $1 \leq i \leq n$, so isomorphic to $\alpha_p$ or
$\mu_p$. To prove the lemma it suffices by \cite[IV, \S 1, 4.5]{DG} to
show that all these quotients are of the same type.

If $G$ is neither unipotent nor multiplicative, it follows from the
above that $G$ has a subquotient of order $p^2$ which is an extension
of $\mu_p$ by $\alpha_p$ or an extension of $\alpha_p$ by $\mu_p$.

In the first case, \cite[IV, \S 2, 3.3]{DG} implies that the extension
is a trigonalisable group scheme. Then by \cite[IV, \S 2, 3.5]{DG} the
extension splits, i.e., it is a semidirect product.
% {\color{red} This  part of the argument is correct.}
In the second case, since $k$ is algebraically closed, so perfect, it
follows from \cite[Th\'eor\`eme 6.1.1 (B)]{SGA3II} that the extension
is trivial.

% {\color{red} This part of the argument is wrong: the cited theorem is
%   about extensions of multiplicative groups by unipotent groups.}

Lemma \ref{lem:sub} leads to a contradiction in both cases, so the
proposition is proved.
\end{proof}

\begin{rem}
  Multiplicative group schemes $G$ with $\dim_k(\mg{g}) = 1$ are just
  forms of $\mu_{p^n}$ for some $n>0$. Can one classify all unipotent
  group schemes with one dimensional Lie algebras? (It seems likely
  that they are all commutative.)
\end{rem}

\section{Proof of the main theorem}

\begin{proof}[Proof of Theorem \ref{thm:edone}]
  For any non-trivial finite group scheme $G$, $\ed_k(G) \geq 1$ since
  there exist $G$-torsors over any extension of $k$ whose underlying
  scheme is integral.

  Now suppose $G$ is an arbitrary finite group scheme with
  $\ed_k(G) = 1$. By Lemma \ref{lem:action}, there exists a smooth
  projective rational variety $X$ with a generically free action of
  $G$. Let $\emptyset \neq U \subset X$ be an open subset on which $G$
  acts freely and let $V = X/G$, so $U$ is the total space of a
  $G$-torsor over $V$. The induced torsor over $\spec(k(V))$ does not
  have essential dimension $0$, since any $G$-torsor over $\spec(k)$
  becomes trivial over an algebraic closure $\bar{k}$ of $k$ but the
  torsor over $\spec(\bar{k}(V_{\bar{k}}))$ is non-trivial. Thus,
  since $\ed_k(G) = 1$, there is a normal integral curve $Z$ over $k$,
  a $G$-torsor $T$ over $Z$, and a dominant rational map $V \dar Z$
  such that $X$ is generically equal to $V \times_Z T$ (which makes
  sense over the generic point of $V$) as a $G$-torsor.

  Since $X$ is integral, the fibre of $T$ over the generic point of
  $Z$ must also be integral. Let $Y$ be the unique normal projective
  curve with function field equal to $k(T)$ (the residue field at the
  generic point of $T$), so $Y$ has a generically defined, generically
  free, action of $G$ and we have a $G$-equivariant dominant rational
  map from $X$ to $Y$.  Since $X$ is rational, Luroth's theorem
  implies that $Y \cong \P^1_k$; in particular, $Y$ is smooth. By
  Proposition \ref{prop:extend}, the generic $G$-action extends to a
  regular action of $G$ on $Y$. The generic freeness of the $G$-action
  on $Y$ then gives an embedding
  $G \hookrightarrow \aut(Y) \cong PGL_{2/k}$. Finally, since the
  $G$-action on $Y$ is generically free and $Y$ is a smooth curve, it
  follows that $\dim_k(\mg{g}) \leq 1$.

  Now let $G \subset PGL_{2/k}$ be an infinitesimal subgroup scheme
  with $\dim_k(\mg{g}) = 1$.  If $G$ lifts to a subgroup scheme of
  $GL_{2/k}$, then $G$ acts generically freely on $\A^2_k$ and the map
  $\A^2_k \to \A^2_k/G$ gives rise to a versal $G$-torsor over a
  non-empty open subset of $\A^2_k/G$. We also have a $G$-torsor
  corresponding to the quotient map $\P^1_k \to \P^1_k/G$, and this
  torsor gives a one dimensional compression of the versal
  $G$-torsor. Thus $\ed_k(G) \leq 1$ and if $G$ is non-trivial then we
  must have $\ed_k(G) =1$.

  To prove that $G$ lifts to $GL_{2/k}$ if $\ed_k(G) = 1$, we first
  consider the case that $\ch(k) \neq 2$. Then we have an etale
  isogeny $SL_{2/k} \to PGL_{2/k}$, so if $G$ is infinitesimal it
  lifts (uniquely) to $SL_{2/k}$, therefore to $GL_{2/k}$.

  Suppose that $G$ is multiplicative (with
  $\ch(k)$ arbitrary) and let $C(G)$ be the centralizer of
  $G$ in
  $PGL_{2/k}$. By \cite[XI, Corollaire 2.4]{SGA3II}, the centralizer
  of a group scheme of multiplicative type in any smooth affine group
  scheme is smooth, so $C(G)$ is smooth; let
  $C^0(G)$ be its identity component. Now
  $C(G)$ is a proper subgroup scheme of
  $PGL_{2/k}$ since the centre of $PGL_{2/k}$ is trivial. Also,
  $G$ is contained in the centre of
  $C^0(G)$, which must therefore be positive dimensional. If
  $C^0(G)$ is not a torus then it must contain a smooth one
  dimensional unipotent subgroup $U$. However, since $G \subset
  G^0(G)$ and
  $G$ is multiplicative, it follows that in this case
  $C^0(G)$ must be two dimensional. Therefore,
  $C^0(G)$ must be a Borel subgroup of $PGL_{2/k}$ and
  $U$ must be its unipotent radical. However, the centre of a Borel
  subgroup of
  $PGL_{2/k}$ is trivial, so this contradicts the fact that
  $G$ is contained in the centre of
  $C^0(G)$.  Thus, any infinitesimal multiplicative subgroup
  $G$ of
  $PGL_{2/k}$ in any characteristic must be contained in a torus $T$.

  Let $T'$ be the inverse image of $T$ in $GL_{2/k}$, so $T'$ is a two
  dimensional torus. It is a maximal torus, of which the non-split
  ones are classified by separable quadratic extensions of $k$: the
  action of the Galois group on the character group $\cong \Z^2$ is
  given by switching coordinates. It follows that $T'$ is a quadratic
  twist of $\G_m$.  If $\ch(k) > 2$ this gives the claimed
  classification of all multiplicative $G$ with $\ed_k(G) = 1$.

  If $T$ is non-split then the above description shows that it
  contains a unique subgroup scheme of order $2$.  If $\ch(k) = 2$, it
  follows from this that $G \subset T'$ lifts to a subgroup of
  $GL_{2/k}$ iff $T$ (equivalently $T'$) is split; in particular, $G$
  must be isomorphic to $\mu_{2^n}$ for some $n$.

  By \cite[Proposition 6.1]{LMMR}, if $G$ is a non-trivial quadratic
  twist of $\mu_{2^n}$ then $\ed_k(G;2) = 2$. Since
  $\ed_k(G) \geq \ed_k(G;p)$ for any prime $p$, we conclude that if
  $\ch(k) = 2$ the only infinitesimal multiplicative group schemes $G$
  with $\ed_k(G) = 1$ are the $\mu_{2^n}$ for $n>0$.

  Now suppose $G \neq \{1\}$ is unipotent and assume that the
  $G$-action on $\P^1_k$ gives a versal $G$-torsor.
  \begin{claim*}
    $G$ preserves a point in $\P^1_k(k)$.
  \end{claim*}
  \begin{proof}[Proof of claim]
  Consider a faithful representation of $G$ on a finite dimensional
  $k$-vector space $V$. Since $G$ is unipotent, by \cite[IV, \S 2,
  2.5]{DG} $V$ has a complete filtration by $k$-subspaces
  \[
    0 = V_0 \subset V_1 \subset \dots\subset V_{n-1} \subset V_n = V
  \]
  which is preserved by $G$. By induction on $n$, we may assume that
  the action on $V_{n-1}$ is not faithful. 

  Since $\lie(G)$ is one dimensional, the action of $G$ on $V$
  (which we now think of as an affine variety) is generically free.
  By the assumed versality of the $G$-action on $\P^1_k$, there exists
  a rational $G$-equivariant morphism $f:V \dar \P^1_k$. Since
  $\P^1_k$ is proper, $f$ is defined at the generic point of
  $V_{n-1}$. The restriction of $f$ to $V_{n-1}$ cannot be dominant
  since the action of $G$ on $V_{n-1}$ is not faithful, so
  $f(V_{n-1})$ must be a rational point $x \in \P^1_k(k)$. Since
  $V_{n-1}$ is $G$-invariant, it follows that so is $x$, and the claim
  is proved.
  \end{proof}
  The stabilizer in $PGL_{2/k}$ of any point in $\P^1_{k}(k)$ is Borel
  subgroup so $G$ is contained in a Borel $B$. Since $G$ is unipotent,
  it is in fact contained in the unipotent radical of $B$ which lifts
  to $GL_{2/k}$. Thus, $G$ also lifts to $GL_{2/k}$.
  \footnote{When $p=2$ there do exist infinitesimal unipotent
    subgroups of $PGL_{2/k}$ which do not lift to $GL_{2/k}$; this was
    pointed out to us by D.~Tossici.}

  Any unipotent subgroup of $GL_{2/k}$ in any characteristic preserves
  a line, so we see that any infinitesimal
  unipotent group $G$ with $\ed_k(G) = 1$ must be isomorphic to a
  subgroup of $\G_a$, so must be $\alpha_{p^n}$ for some $n>0$.
\end{proof}

\begin{rem}
  If $\ch(k) = 2$ and $\mu_2$ is embedded in $PGL_{2/k}$ as a subgroup
  of a torus (which is uniquely determined as the connected
  centralizer), then the $\mu_2$ action on $\P^1_k$ has a fixed point
  iff the torus is split. It follows that the embedding gives a versal
  $\mu_2$-torsor iff the torus is split.
\end{rem}

\section{Applications of the main theorem}

\subsection{Finite group schemes of essential dimension
  one} \label{sec:finite}%

Let $G$ be any finite group scheme over $k$ with $\ed_k(G)=1$. By
Theorem \ref{thm:edone}, $G$ embeds in $PGL_{2/k}$ and we also have
$\dim_k(\mg{g}) = 1$. If $G$ is any finite subgroup scheme of
$PGL_{2/k}$ with $\dim_k(\mg{g}) = 1$, then a sufficient condition for
$\ed_k(G) = 1$ is that $G$ lifts to $GL_{2/k}$. If this condition were
also necessary---we know this is the case for constant as well as
infinitesimal group schemes---then we would have a complete
classification of all finite group schemes $G$ with $\ed_k(G) = 1$.

As a first step, one should verify this for etale group schemes. This
can presumably be done using the classification of constant groups $G$
with $\ed_k(G) = 1$ in \cite{CHKZ}, but we do not do this here and
proceed after making some simplifying assumptions.

Recall that for any finite group scheme $G$ we have an exact sequence
\[
  1 \to G^0 \to G \to G^{et} \to 1 ,
\]
where $G^0$ is infinitesimal and $G^{et}$ is etale. This sequence
splits when $k$ is perfect, i.e., $G \cong G^0 \rtimes G^{et}$
\cite[II, \S 5, 2.4]{DG}.

\begin{thm} \label{thm:finite}%
  A finite group scheme over a perfect field $k$ with $G^{et}$
  constant has $\ed_k(G) \leq 1$ iff $G$ can be embedded in $PGL_{2/k}$,
  $\dim_k(\mg{g}) \leq 1$ and $G$ lifts to $GL_{2/k}$.
\end{thm}
Note that if $k$ is algebraically closed this gives a classification
of all finite group schemes over $k$ with $\ed_k(G) = 1$. For general
perfect $k$ the method of proof, together with the classification
results of \cite{CHKZ}, can be used to make a more explicit list of
all such group schemes.

\begin{proof}
  The conditions are clearly sufficient and we have already seen that
  the first two are necessary, so we need to show that if
  $\ed_k(G) = 1$ then $G$ has an embedding in $PGL_{2/k}$ which lifts
  to $GL_{2/k}$. We may assume that $G^0$ is non-trivial since
  otherwise the theorem is a consequence of \cite{ledet-one}.

  If $\ed_k(G) = 1$ then there is an embedding
  $G \hookrightarrow PGL_{2/k}$ which makes the quotient map
  $\P^1_k \to \P^1_k/G$ generically into a versal $G$-torsor. This
  implies (as is well known) that $\ed_k(G^0) = 1$ and $\ed_k(G^{et})$
  and the corresponding quotient maps are also versal torsors. By
  Theorem \ref{thm:edone} and \cite[Theorem 8]{ledet-one}, both $G^0$
  and $G^{et}$ lift to $GL_{2/k}$; call the lifts $\wt{G^0}$ and
  $\wt{G^{et}}$.

  If $\ch(k) \neq 2$, we saw in the proof of Theorem \ref{thm:edone}
  that we can assume that $\wt{G^0} \subset SL_{2/k}$ and such a lift
  is unique. It follows that $\wt{G^{et}}$ normalizes $\wt{G^0}$ and
  the subgroup $\wt{G^0} \rtimes \wt{G^{et}} \subset GL_{2/k}$ is a
  lift of $G$.

  If $\ch(k) = 2$ and $G^0$ is unipotent then $\wt{G^0}$ is still
  unique, so again $\wt{G^0} \rtimes \wt{G^{et}} \subset GL_{2/k}$ is
  a lift of $G$. Finally, suppose that $G^0$ is multiplicative, so it
  is isomorphic to $\mu_{2^n}$ for some $n>0$. We may assume that
  $\wt{G^0}$ is a subgroup of the group $T'$ of diagonal matrices in
  $GL_{2/k}$ and $\wt{G^0} \cap Z(GL_{2/k}) = \{1\}$. If
  $\wt{G^0} \neq \{1\}$, one easily sees that the normalizer of
  $\wt{G^0}$ in $GL_{2/k}$ is $T'$. Thus, $G$ has a lift to $GL_{2/k}$
  iff $G \subset T$, the image of $T'$ in $PGL_{2/k}$. If this
  condition is not satisfied then $G$ contains a subgroup
  $G_1 \cong \mu_2 \times \Z/2\Z$. By \cite[Theorem 3.1]{bab-che}
  $\ed_k(\mu_2 \times \Z/2\Z) = 2$---one may give an elementary proof
  of this particular case by using that the action of $G_1$ on
  $\P^1_k$ does not have any fixed points\footnote{This, together with
    the reference in \cite{bab-che}, was pointed out by
    Z.~Reichstein.}--- and this implies that $\ed_k(G) > 1$.

  We conclude that if $\ed_k(G) = 1$ then $G$ lifts to $GL_{2/k}$.
\end{proof}

\subsection{Some finite group schemes of essential dimension two}

Using Theorem \ref{thm:edone} we may compute $\ed_k(G)$ for various
other group schemes. The point is that if one knows that
$\ed_k(G) \leq 2$ for some group scheme $G$ not occurring in the list
of group schemes with $\ed_k = 1$ then we must have $\ed_k(G) = 2$.
\begin{ex} \label{ex:ss}%
  Let $G = E[p]$, the $p$-torsion of a supersingular elliptic curve
  $E$ over a field $k$ with $\ch(k) = p>0$. It sits in an exact sequence
  \[
    0 \to \alpha_p \to E[p] \to \alpha_p \to 0 ,
  \]
  so by \cite[Theorem 1.4]{tos-vis} we have $\ed_k(E[p]) \leq 2$.
  However, $E[p]$ is not isomorphic to $\alpha_{p^2}$, so we conclude
  that $\ed_k(E[p]) = 2$. The group scheme $E[p]$ is trigonalizable,
  but $2 = \ed_k(E[p]) > \dim_k(\lie(E[p])) = 1$, so it is not
  \emph{almost special} in the sense of \cite[Definition
  4.2]{tos-vis}. This answers a question discussed in \cite[Example
  4.8]{tos-vis}.
\end{ex}

\def\cprime{$'$}

\end{document}